\documentclass[12pt]{amsart}

\usepackage{url}
\usepackage{amssymb, amsmath}
\usepackage{enumerate}
\usepackage{braket}
\usepackage{bbm}
\usepackage{bm}
\usepackage{mathrsfs}
\usepackage[top=30truemm, bottom=30truemm, left=25truemm, right=25truemm]{geometry}

\theoremstyle{definition}
\newtheorem{definition}{Definition}[section]

\newtheorem{remark}[definition]{Remark}

\newtheorem*{ackname}{Acknowledgment}

\theoremstyle{theorem}
\newtheorem{theorem}[definition]{Theorem}
\newtheorem{corollary}[definition]{Corollary}
\newtheorem{lemma}[definition]{Lemma}
\newtheorem{proposition}[definition]{Proposition}

\renewcommand{\phi}{\varphi}
\renewcommand{\Re}{\mathrm{Re}\hspace{1pt}}
\renewcommand{\Im}{\mathrm{Im}\hspace{1pt}}
\renewcommand{\epsilon}{\varepsilon}
\renewcommand{\phi}{\varphi}

\DeclareMathOperator{\meas}{meas}
\DeclareMathOperator{\dist}{dist}

\newcommand{\te}{\tilde{\eta}}
\newcommand{\EXP}{\mathbb{E}}
\newcommand{\Li}{\operatorname{Li}}
\newcommand{\AND}{\quad {\rm{and}} \quad}

\begin{document}
\title[Universality theorem for the iterated integrals of the logarithm of $\zeta(s)$]{Universality theorem for the iterated integrals of the logarithm of the Riemann zeta-function}

\author[K.~Endo]{Kenta Endo}

\subjclass[2020]{11M41}

\keywords{Riemann zeta-function, Universality theorem for zeta-functions.}

\maketitle

\begin{abstract}
In this paper, 
we prove the universality theorem for the iterated integrals of the logarithm of the Riemann zeta-function on some line parallel to the real axis.
\end{abstract}

\section{Introduction and main results}
Let $\zeta(s)$ denote the Riemann zeta-function. 
In this paper, 
we consider the function $\te_m(s)$ which is given by
\[
\te_m(\sigma + it) = \int_{\sigma}^{\infty} \te_{m -1} (\alpha + it) d\alpha \quad \textrm{and} \quad 
\te_{0}(\sigma + it) = \log \zeta (\sigma + it) = \int_{\infty}^{\sigma} \frac{\zeta'}{\zeta} (\alpha + it) d\alpha
\]
for $s = \sigma + it \in G$, 
where $G$ is defined by
\[
G 
=\mathbb{C} \setminus \left\{ \left( \bigcup_{\rho = \beta + i \gamma} \left\{ s = \sigma + i \gamma; \sigma \leq \beta \right\}  \right) \cup  (- \infty, 1]   \right\}.
\]
Here $\rho = \beta + i \gamma$ denote the non-trivial zeros of the Riemann zeta-function $\zeta(s)$.
This function $\te_m (s)$ is firstly introduced by Inoue \cite{I2019} who investigated various properties of it.
Recently, he and the author \cite{EI2020} showed that the set $\left\{ \te_m( \sigma + it); t \in \mathbb{R} \right\}$
is dense in the complex plane for $m \geq 1$ and $1/2 \leq \sigma <1$ by using the mean value theorem in \cite[Theorem 5]{I2019}. 
We note that Bohr \cite{B1916} proved classically that the set $\left\{ \log \zeta(\sigma+ it) ~;~ t \in \mathbb{R} \right\}$ is dense in the complex plane for $1/2 < \sigma < 1$,  
and it remains to be determined whether or not the set $\left\{ \log \zeta(1/2+ it) ~;~ t \in \mathbb{R} \right\}$ is dense in the complex plane. 

Historically, Bohr's result has been developed to the Bohr-Jessen limit theorem \cite{BJ1930, JW1935}, Voronin's multidimensional denseness theorem \cite{V1972} and Voronin's universality theorem \cite{V1975}.
Our goal in this paper is to show that the universality theorem for $\te_m(s)$ in the strip $\mathcal{D} = \{s = \sigma + it ~;~ 1/2 < \sigma < 1\}$.
We refer to the textbooks \cite{KV1992, L1996, S2007, K2015} for the theory of the universality theorem and refer to a survey paper \cite{M2015} for the recent studies. 

For other recent studies on $\te_m(s)$, 
we remark that Inoue \cite{I2020} studied on the extreme values and the large deviation estimate like \cite{L2011}.
Inoue, Mine and the author \cite{EIM2021} studied on the probability density function, the discrepancy and the large deviation estimates like \cite{BJ1930, LLR2019}.

Now, we state the main theorem of the present paper. 

\begin{theorem}\label{thm:MTH}
Let $m$ be a non-negative integer, $\mathcal{K}$ be a compact subset of $\mathcal{D}$ with connected complement, 
and let $f(s)$ be a continuous function on $\mathcal{K}$ that is holomorphic in the interior of $\mathcal{K}$.
Then for any $\epsilon>0$ we have
\begin{equation}\label{eqn:UTE}
\liminf_{T \rightarrow \infty} \frac{1}{T} \meas\left\{ \tau \in [T,2T] ~;~ \sup_{s \in \mathcal{K}} \left| \te_{m} (s + i \tau) - f(s) \right| < \epsilon \right\}>0,
\end{equation}
where $\meas$ denotes the Lebesgue measure.
\end{theorem}

We remark that, in the set lying on the left hand side of \eqref{eqn:UTE}, 
the set of all $\tau \in [0,T]$ such that the function $\te_m(s + i \tau)$ is not holomorphic for some $s \in \mathcal{K}$ is excluded. 

Using this theorem, 
we have the following corollaries.

\begin{corollary}\label{cor:1}
Let $m$ be a non-negative integer, $n$ be a positive integer and $1/2 < \sigma <1$.
Then the set
\[
\left\{ \left( \te_m( \sigma + it), \te_m'( \sigma + it), \ldots, \te_m^{(n-1)}( \sigma + it) \right) \in \mathbb{C}^n~;~ t \in \mathbb{R} \right\}
\]
is dense in $\mathbb{C}^n$.
\end{corollary}

\begin{corollary}\label{cor:2}
Let $m$ be a non-negative integer, $n$ be a positive integer and $1/2 < \sigma <1$.
If $F_0, F_1, \ldots , F_J  : \mathbb{C}^n \rightarrow \mathbb{C}$ are continuous functions and satisfy
\[
\sum_{j = 0}^J s^j F_j \left( \te_m( s ), \te_m'( s), \ldots, \te_m^{(n-1)}(s) \right) \equiv 0
\]
for all $s \in \mathbb{C}$, then $F_0, F_1, \ldots , F_J \equiv 0$.
\end{corollary}

These corollaries can be proved along the same line as in \cite[Chapter VII $\S 2$]{KV1992}.
Therefore we omit the proofs.

\section{Proof of Theorem \ref{thm:MTH}}
\subsection{Preliminaries}
We begin with giving some definitions and notations.
Let us fix a compact subset $\mathcal{K}$ of $\mathcal{D}$ with connected complement. 
Since, without the assumption of the Riemann hypothesis, 
we can not find whether or not the function $\te_m(s + i \tau)$ is always holomorphic on $s \in \mathcal{K}$, 
we need to use the shift $\tau$ which implies the function $\te_m(s + i \tau)$ is holomorphic on $s \in \mathcal{K}$. 
For such a reason, 
we first prepare the following notations; 
\begin{itemize}
\item $| \mathcal{K}  |: = \max_{s \in \mathcal{K}} \Im(s) - \min_{s \in \mathcal{K}} \Im(s)$. 
\item $\tau_0(\mathcal{K}) = 1/2( \max_{s \in \mathcal{K}} \Im(s) + \min_{s \in \mathcal{K}} \Im(s) )$.
\item $\sigma_0(\mathcal{K})= 1/2 (1/2 + \min_{s \in \mathcal{K}} \Re(s))$.
\item For any $\Delta >0$, put 
\begin{equation}\label{eqn:HR}
\mathcal{G}_{\sigma_0(\mathcal{K}),\Delta}
= \mathbb{R} \setminus \left\{ \left( \bigcup_{\substack{\rho = \beta + i \gamma;\\ \beta > \sigma_0(\mathcal{K})}} \left(\gamma - \tau_0(\mathcal{K}) - \Delta , \gamma - \tau_0(\mathcal{K}) + \Delta \right) \right) \cup (- \tau_0(\mathcal{K}) - \Delta, - \tau_0(\mathcal{K}) + \Delta) \right\}. 
\end{equation}
\item For any $T>0$, 
let $\mathcal{I}_{\mathcal{K}}(T) =  \mathcal{G}_{\sigma_0(\mathcal{K}), |\mathcal{K}| +1}  \cap [T, 2T]$.
\end{itemize}
Note that $\tau \in \mathcal{G}_{\sigma_0(\mathcal{K}), |\mathcal{K}| +1}$ implies $\mathcal{K} + i \tau \subset G$.
Thus the function $\te_m(s + i \tau)$ is holomorphic on $s \in \mathcal{K}$ when $\tau \in \mathcal{G}_{\sigma_0(\mathcal{K}), |\mathcal{K}| +1}$. 
We further note that $\meas (\mathcal{I}_{\mathcal{K}}) \sim T$ holds as $T \rightarrow \infty$ by using the zero density estimate. 
Remark that, assuming the Riemann Hypothesis, $\mathcal{I}_{\mathcal{K}} (T) = [T, 2T]$ holds.

Let $\gamma$ denote the unit circle on the complex plane 
and define $\Omega$ by $\Omega = \prod_{p} \gamma_p$, where the product runs over all prime numbers $p$ and $\gamma_p = \gamma$. 
Since $\gamma_p$ is a compact topological abelian group,
there exists the unique probability Haar measure $\mathbf{m}_p$ on $(\gamma_p, \mathcal{B}(\gamma_p))$.
Here, $\mathcal{B}(S)$ denote the Borel $\sigma$-field of the topological space $S$.
Using the Kolmogorov extension theorem, we can construct the probability Haar measure $\mathbf{m} = \otimes_p \mathbf{m}_p$ on $(\Omega, \mathcal{B}(\Omega))$.
For any $\omega = ( \omega(p) )_{p} \in \Omega$ and $n\in \mathbb{N}$ with $n \geq 2$, let
\[
\omega(n) = \prod_{j = 1}^k \omega(p_j)^{r_j},
\]
where $n = p_1^{r_1} \cdots p_k^{r_k}$ is the prime factorization of $n$.

Fix a non-negative integer $m$ and let $T$ be a positive real number. 
We fix real numbers $\sigma_L = \sigma_L(\mathcal{K})$ and $\sigma_R = \sigma_R(\mathcal{K})$ satisfying
\[
\sigma_0(\mathcal{K}) < \sigma_L < \min_{s \in \mathcal{K}}\Re (s) 
\quad \textrm{and} \quad 
\max_{s \in \mathcal{K}} \Re(s) < \sigma_R < 1,
\]
and denote by $\mathcal{R} = \mathcal{R} (\mathcal{K})$ the rectangle
\begin{equation}\label{eqn:RECT}
\mathcal{R}
= (\sigma_L, \sigma_R) \times i \left( \min_{s \in \mathcal{K}} \Im (s) - 1 /2, \max_{s \in \mathcal{K}} \Im(s) + 1/2 \right),
\end{equation}
which includes $\mathcal{K}$. 
Note that $\tau \in \mathcal{G}_{\sigma_0(\mathcal{K}), |\mathcal{K}| +1}$ implies $\mathcal{R} + i \tau \subset G$ by the definition of $\mathcal{R}$. 
Let $\mathcal{H}(\mathcal{R})$ denote the set of holomorphic functions on $\mathcal{R}$ equipped with the topology of uniform convergence on compact subsets.
This topology is metrizable in the following way;
We take a sequence of compact sets $\{K_j\}_{j =1}^\infty$ satisfying
\begin{itemize}
\item $\mathcal{R} = \bigcup_{j =1}^\infty K_j$,
\item $K_j \subset K_{j+1}^{\circ}$ for any $j \in \mathbb{N}$, 
where $A^{\circ}$ is the interior of the set $A$,
\item for any compact subset $K$, there exist $j \in \mathbb{N}$ such that $K \subset K_j$,
\end{itemize}
and define the metrics $d_j$, $j = 1, 2, \ldots$ on $\mathcal{H}(\mathcal{R})^2$ by $d_j (f, g) = \sup_{s \in K_j} \left| f(s) - g(s) \right|$ for $f, g \in \mathcal{H}(\mathcal{R})$.
For the metric sequence $\{d_j\}_{j =1}^\infty$, 
we put 
\[
d (f, g)
= \sum_{j =1}^\infty \frac{\min\left\{ d_j ( f, g ) , 1\right\}}{2^j}
\]
for $f, g \in \mathcal{H}(\mathcal{R})$.
We find that $d$ is a metric on $\mathcal{H}(\mathcal{R})$ and induces the desired topology.

We define the probability measures $\mathcal{Q}_T$ and $\mathcal{Q}$ on $(\mathcal{H}(\mathcal{R}), \mathcal{B}(\mathcal{H}(\mathcal{R})))$ by
\begin{align*}
\mathcal{Q}_T (A)
= \frac{1}{\meas\left(\mathcal{I}_{\mathcal{K}}(T) \right)} \meas \left\{ \tau \in \mathcal{I}_{\mathcal{K}}(T)  ; \te_m (s + i \tau ) \in A \right\},
\end{align*}
\[
\mathcal{Q}(A)
= \mathbf{m}\left\{ \omega \in \Omega ; \te_m(s, \omega) \in A \right\}
\]
for $A \in \mathcal{B}(\mathcal{H}(\mathcal{R}))$,
where $\mathcal{H}(\mathcal{R})$-valued random variables $\te_m(s,\omega)$ is defined by
\[
\te_m(s,\omega)
= \sum_{n =2}^\infty \frac{\Lambda (n) \omega (n) }{n^{s} (\log p)^{m +1} }.
\] 
Here $\Lambda(n)$ denotes the von Mangoldt function given by 
\[
\Lambda(n)
= \begin{cases}
\log p & \textrm{if $n$ is a power of a prime number $p$},\\ 
 0 & \textrm{otherwise}.
\end{cases}
\]

We shall show the following two propositions.
\begin{proposition}\label{prop:LT}
The probability measure $\mathcal{Q}_T$ on $(\mathcal{H}(\mathcal{R}), \mathcal{B}(\mathcal{H}(\mathcal{R})))$ converges weakly to $\mathcal{Q}$ as $T \rightarrow \infty$.
\end{proposition}

\begin{proposition}\label{prop:SP}
The support of the probability measure $\mathcal{Q}$ on $(\mathcal{H}(\mathcal{R}), \mathcal{B}(\mathcal{H}(\mathcal{R})))$ coincides with $\mathcal{H}(\mathcal{R})$.
\end{proposition}

Throughout the paper, 
we use the following convention; 
For a probability space $(X, \mathcal{F}, \mu)$ and an integrable function $f$ on $X$, 
we write $\mathbb{E}^{\mu} \left[ f \right] = \int_{X} f d\mu$.

\subsection{Proof of Proposition \ref{prop:LT}}
We will prove the proposition by using the modern method as in \cite{K2015} and \cite{K2017}.
Although the basic strategy of the proof is based on the method of Bagchi in \cite{B1981}, 
we note that one need not to use the Birkhoff-Khinchin ergodic theorem in this way.

To show the proposition, we use the smoothing technique in analytic number therory.
We fix a real valued smooth function $\phi(x)$ on $[0,\infty)$ with compact support satisfying $\phi(x) = 1$ on $[0,1]$ and $0 \leq \phi(x) \leq 1$ for all $x \in \mathbb{R}$.
Let $\hat{\phi} (s)$ denote the Mellin transform of $\phi(x)$, i.e. $\hat{\phi} (s)= \int_{0}^\infty \phi(x) x^{s -1} dx$ for $\Re(s)>0$.
We recall the basic properties of the Mellin transform.

\begin{lemma}\label{lem:MT}
We have the following;
\begin{itemize}
\item[(i)] The Mellin transform $\hat{\phi} (s)$ has the meromorphic continuation on  $\Re(s) > -1$ with at most a simple pole at $s =0$ with residue $\phi(0) = 1$.
\item[(ii)] Let $-1< A < B$ be real numbers.
Then for any positive interger $N>0$, there exists $C(A,B;N) > 0$ such that 
\[
\left| \hat{\phi} (s) \right| 
\leq C(A,B;N) ( 1 + |t| )^{-N}
\]
holds for $s = \sigma + it$ with $A \leq \sigma \leq B $ and $|t|\geq 1$.
\item[(iii)] For any $c >0$ and $x > 0$, we have the Mellin inversion formula
\[
\phi(x) = \frac{1}{2 \pi i} \int_{c - i \infty}^{c + \infty} \hat{\phi} (s) x^{- s}ds.
\]
\end{itemize}
\end{lemma}

\begin{proof}
The proof can be found in \cite[Appendix A]{K2017}.
\end{proof}

We use the following lemmas which are a little modified versions of Lemma 2.1 and Lemma 2.2 in a paper of Granville and Soundararajan \cite{GS2006}.

\begin{lemma}\label{lem:FGS}
Let $y \geq 2$ and $|t|\geq y + 3$ be real numbers.
Let $1/2 \leq \sigma_0 < 1$ and suppose that the rectangle $\{z ~;~ \sigma_0 < \Re(s) \leq 1,\, | \Im(z) - t | \leq y + 2 \}$ is free of zeros of $\zeta(z)$. 
Then for $\sigma_0 < \sigma \leq 1$ we have
\[
\te_m( \sigma + it ) = \sum_{2 \leq n \leq y} \frac{\Lambda(n)}{n^{\sigma + it} ( \log n )^{m +1}} + O_m\left( \frac{\log |t|}{(\sigma_1 - \sigma_0)^2} y^{\sigma_1 - \sigma} \right),
\]
where we put $\sigma_1 = \min\left\{ \sigma_0 + 1/\log y, (\sigma + \sigma_0)/2 \right\}$. 
\end{lemma}

\begin{proof}
For the case $m =0$, the statement has been proved in Lemma 2.1 in \cite{GS2006}.
For the other cases, the proof can be done in the same way.
\end{proof}

\begin{lemma}\label{lem:GS}
Let $1/2 < \sigma_0 \leq 1$ be fixed and let $T$ and $y$ be large numbers with $T \geq y + 3$.
Put 
\begin{align*}
\bm{\ell}\left([ (1/2)T, (5/2) T ];y\right) 
=& \left( \bigcup_{\substack{\rho = \beta + i \gamma;\\
\beta > (1/2)( 1/2 + \sigma_0 ),\\
\gamma \in [ (1/2)T, (5/2) T ]}} \left( \gamma - (y + 3),  \gamma + ( y +3 )\right)\right) \\
& \quad \cup [(1/2)T, (1/2)T + (y + 3) ] 
\cup [ (5/2)T - (y + 3), (5/2)T]. 
\end{align*}
Then we have
\[
\te_m( \sigma + i t ) 
= \sum_{2 \leq n \leq y} \frac{\Lambda(n)}{n^{\sigma + it} (\log n)^{m +1} } 
+ O_{m}( y^{(1/2 - \sigma_0)/2} (\log T)^3 )
\]
for $\sigma_0 \leq \sigma \leq 1$ and $t \in [ (1/2)T, (5/2) T ] \setminus \bm{\ell}\left([ (1/2)T, (5/2) T ];y\right)$, and the estimate
\[
\meas \left( \bm{\ell}\left([ (1/2)T, (5/2) T ];y\right) \right) \ll T^{5/4 - \sigma_0/2 } y (\log T)^5
\]
holds.
\end{lemma}

\begin{proof}
The first assertion can be obtained by applying Lemma \ref{lem:FGS} with $\sigma_0'=1/2(1/2 + \sigma_0)$ in place of $\sigma_0$. 
The last assertion follows from the zero-density estimate $N(\sigma_0', T) \ll T^{3/2 - \sigma_0'}(\log T)^5$ (see e.g. \cite[Theorem 9.19 A]{T1986}). 
\end{proof}

We will use Lemma \ref{lem:GS} with $\sigma_0 = \sigma_0(\mathcal{K})$ in the next lemma. 
In what follows, let 
\[
\mathscr{X}_{\mathcal{K}} ( T ) 
= \mathcal{G}_{\sigma_0(\mathcal{K}),\left| \mathcal{K} \right|+Y(T) + 4}  \cap [T, 2T] \quad \textrm{and} \quad Y(T) = (\log T)^{8/(\sigma_0(\mathcal{K}) - 1/2 )}
\]
and define the function
\[
\te_{m,X}(s + i \tau)
:= \sum_{n=2}^\infty \frac{\Lambda(n) \phi ( n / X ) }{n^{s + i \tau} (\log n)^{m + 1} }
= \sum_{2 \leq n \leq X} \frac{\Lambda(n) }{n^{s + i \tau} (\log n)^{m + 1} }
+ R_{m,X}(s + it)
\]
for $X \geq 2$.
Note that $\mathscr{X}_{\mathcal{K}} ( T ) \subset I_\mathcal{K}(T)$ holds, and $\meas\left(\mathscr{X}_{\mathcal{K}} ( T )\right) \sim T$ holds as $T \rightarrow \infty$ by the zero-density estimate.
Then we have the following lemma.

\begin{lemma}\label{lem:LT1}
For any compact subset $C$ of $\mathcal{R}$, we have
\begin{align*}
\lim_{X \rightarrow \infty} \limsup_{T \rightarrow \infty} \frac{1}{\meas\left( \mathcal{I}_{\mathcal{K}}(T) \right)}
\int_{\mathscr{X}_{\mathcal{K}} ( T )}  \sup_{s \in C} \left| \te_m(s + i \tau) - \te_{m,X}(s + i \tau) \right| d \tau 
= 0.
\end{align*}
\end{lemma}

\begin{proof}[Proof of Lemma \ref{lem:LT1}]
Let $T$ be large number.
We note that, for any $\tau \in \mathscr{X}_{\mathcal{K}}( T )$, 
the function $\te_m(s + i\tau)$ is holomorphic on $\mathcal{R}$. 
By Cauchy's formula, 
\[
\te_m(s + i \tau) - \te_{m,X}(s + i \tau)
= \frac{1}{2 \pi i} \int_{\partial \mathcal{R}} \frac{\te_m(z + i \tau) - \te_{m,X}(z + i \tau)}{z -s} dz
\]
holds for $s \in C$ and $\tau \in \mathscr{X}_{\mathcal{K}}( T )$.
Here $\partial A$ denotes the boundary of the set $A$.
Hence we have
\begin{align*}
&\frac{1}{\meas\left( \mathcal{I}_{\mathcal{K}}(T) \right)}
\int_{\mathscr{X}_{\mathcal{K}} ( T ) }  \sup_{s \in C} \left| \te_m(s + i \tau) - \te_{m,X}(s + i \tau) \right| d \tau\\
\ll& \frac{1}{\dist(C, \partial \mathcal{R}) \meas\left( \mathcal{I}_{\mathcal{K}}(T) \right)}
\int_{\partial \mathcal{R}} |dz| \int_{\mathscr{X}_{\mathcal{K}} ( T )} \left| \te_m(z + i \tau) - \te_{m,X}(z + i \tau) \right| d \tau \\
\ll& \frac{\ell(\partial \mathcal{R})}{\dist(C, \partial \mathcal{R}) \meas\left( \mathcal{I}_{\mathcal{K}}(T) \right)}
\sup_{\sigma; s \in \partial \mathcal{R}} \int_{\mathscr{X}'_{\mathcal{K}}(T)} \left | \te_m(\sigma + i t) - \te_{m,X}(\sigma + i t) \right| dt,
\end{align*}
where $\dist(C, \partial \mathcal{R})$ denotes the distance between $C$ and $\partial \mathcal{R}$ and $\ell(\partial \mathcal{R})$ does the length of $\partial \mathcal{R}$, 
and we put 
\[
\mathscr{X}'_{\mathcal{K}}(T)
= [ (1/2)T, (5/2) T ] \setminus \bm{\ell}\left([ (1/2)T, (5/2) T ];Y(T)/2\right).
\]
Now it holds that
\begin{align*}
&\int_{\mathscr{X}'_{\mathcal{K}}(T)}\left | \te_m(\sigma + i t) - \te_{m,X}(\sigma + i t) \right| dt \\
\leq& \int_{\mathscr{X}'_{\mathcal{K}}(T)}\left | \te_m(\sigma + i t) - \sum_{2 \leq n \leq Y(T)/2} \frac{\Lambda(n)}{n^{\sigma + it} (\log n)^{m +1} } \right| dt \\
&+ \int_{\mathscr{X}'_{\mathcal{K}}(T)}\left | \sum_{2 \leq n \leq Y(T)/2} \frac{\Lambda(n)}{n^{\sigma + it} (\log n)^{m +1} } - \te_{m,X}(\sigma + i t) \right| dt.
\end{align*}
As for the first integral, we have
\[
\int_{\mathscr{X}'_{\mathcal{K}}(T)}\left | \te_m(\sigma + i t) - \sum_{2 \leq n \leq Y(T)/2} \frac{\Lambda(n)}{n^{\sigma + it} (\log n)^{m +1} } \right| dt
\ll \frac{ \meas (\mathscr{X}'(T) ) }{\log T}
\]
by Lemma \ref{lem:GS}.
As for the second integral, we have 
\begin{align*}
&\left( \int_{\mathscr{X}'_{\mathcal{K}}(T)}\left | \sum_{2 \leq n \leq Y(T)/2} \frac{\Lambda(n)}{n^{\sigma + it} (\log n)^{m +1} } - \te_{m,X}(\sigma + i t) \right| dt \right)^2 \\
\leq& \meas(\mathscr{X}'_{\mathcal{K}}(T)) \int_{T/2}^{(5/2)T} \left | \sum_{2 \leq n \leq Y(T)/2} \frac{\Lambda(n)}{n^{\sigma + it} (\log n)^{m +1} } - \te_{m,X}(\sigma + i t) \right|^2  dt
\end{align*}
by using the Cauchy-Schwarz inequality. 
The latter mean square value is estimated by
\begin{align*}
&\int_{T/2}^{(5/2)T} \left | \sum_{2 \leq n \leq Y(T)/2} \frac{\Lambda(n)}{n^{\sigma + it} (\log n)^{m +1} } - \te_{m,X}(\sigma + i t) \right|^2  dt \\
\ll & \left( \sum_{n > X}  \frac{\Lambda (n)^2\left ( 1 - \phi(n /X) \right)^2}{n^{2 \sigma_0 (\mathcal{K})} (\log n)^{2 m + 2} } \right) T + E_1, 
\end{align*}
where
\begin{align*}
E_1
\ll& \sum_{X < l < n \leq Y(T)/2} \frac{ \Lambda (l) \Lambda(n) ( 1 - \phi (l/X) ) ( 1 - \phi (n/X) ) }{l^\sigma n^\sigma ( \log l )^{m +1} ( \log n )^{m +1} \log (n/l) } \\
\ll&_{\sigma_0(\mathcal{K})} Y(T)^{2 - 2 \sigma_0(\mathcal{K})} \log (Y(T)). 
\end{align*}
Combining the above estimates with the estimate $\meas\left( \mathcal{I}_{\mathcal{K}}(T) \right) \sim T$, 
we obtain
\begin{align*}
&\limsup_{T \rightarrow \infty} \frac{1}{\meas\left( \mathcal{I}_{\mathcal{K}}(T) \right)}
\int_{\mathscr{X}_{\mathcal{K}}(T) }  \sup_{s \in C} \left| \te_m(s + i \tau) - \te_{m,X}(s + i \tau) \right| d \tau \\
\ll&_{\mathcal{K}, \mathcal{R}, C} \left( \sum_{n > X}  \frac{\Lambda (n)^2}{n^{2 \sigma_0(\mathcal{K})} (\log n)^{2 m + 2} } \right)^{1/2},
\end{align*}
and the right hand side of the above inequality tends to $0$ as $X \rightarrow \infty$ since $2 \sigma_0(\mathcal{K}) >1$.
This completes the proof.
\end{proof}

Next, we define the  $\mathcal{H}(\mathcal{R})$-valued random variables $\te_{m,X}(s,\omega)$ by 
\[
\te_{m,X}(s,\omega)
= \sum_{n=2}^\infty \frac{\Lambda(n) \omega(n) \phi ( n / X ) }{n^{s} (\log n)^{m + 1} }.
\]
We will prove the following lemma.

\begin{lemma}\label{lem:LT2}
For any compact subset $C$ of $\mathcal{R}$, we have
\[
\lim_{X \rightarrow \infty } \EXP^{\mathbf{m}} \left[ \sup_{s \in C} \left| \te_m(s, \omega) - \te_{m,X} (s, \omega) \right| \right] 
=0.
\]
\end{lemma}

To prove this lemma, we prepaer the following lemma.

\begin{lemma}\label{lem:KE}
We have the following; 
\begin{enumerate}
\item[\rm{(i)}] The series $\sum_{p} \omega(p) p^{ - s} (\log p) ^{ - m} $ is holomorphic on the domain $\Re(s) > \sigma_0(\mathcal{K})$ for almost all $\omega \in \Omega$.
Here the sum runs over all prime numbers. 

\item[\rm{(ii)}]
The function $\te_m(s)$ is holomorphic on the domain $\Re(s) >  \sigma_0(\mathcal{K}) $ for almost all $\omega \in \Omega$.
Moreover,  the equation 
\[
\te_m(s,\omega) = \sum_{p}\frac{\Li_{m +1} ( p^{ - s} \omega(p) ) }{(\log p)^m}
\]
holds on the same region for almost all $\omega \in \Omega$.
Here $\Li_{J} (z)$ means the polylogarithm function defined by $\Li_{J} (z) = \sum_{n =1}^\infty z^n/n^{J}$ for $|z| <1$ and $J \in \mathbb{N}$. 

\item[\rm{(iii)}]
For almost all $\omega \in \Omega$, 
there exists $C_1(\mathcal{K}, \sigma_L, \omega)>0$ such that the estimate 
\[
| \te_m(s, \omega) | \leq C_1(\mathcal{K}, \sigma_L, \omega) (| t | + 2)
\]
holds for $\Re(s) \geq 1/2 (\sigma_0(\mathcal{K}) + \sigma_L)$.
\item[\rm{(iv)}] There exists $C_2(\mathcal{K}, \sigma_L)>0$ such that the estimate
\[
|\EXP^{\mathbf{m}}\left[ \left| \te_m(s, \omega) \right| \right] |
\leq C_2(\mathcal{K}, \sigma_L) ( | t | +2 )
\]
holds for $\Re(s) \geq 1/2 (\sigma_0(\mathcal{K}) + \sigma_L)$.
\end{enumerate}
\end{lemma}

\begin{proof}
We first show that the series $\sum_p \omega (p) p^{- \sigma_0(\mathcal{K})} (\log p)^{ - m}$ converges for almost all $\omega \in \Omega$. 
We have
\[
\EXP^{\mathbf{m}}\left[ \frac{ \omega (p) }{p^{\sigma_0(\mathcal{K})} (\log p)^m }\right]
=0 
\quad \textrm{and} \quad
\sum_{p} \EXP^{\mathbf{m}}\left[ \left|  \frac{ \omega (p) }{p^{\sigma_0(\mathcal{K})} (\log p)^m } \right |^2 \right]
\leq \sum_{n =1}^\infty\frac{1}{n^{2 \sigma_0(\mathcal{K})}} < \infty.
\]
Hence the assertion follows by applying Kolmogorov's theorem (see e.g. \cite[Appendix B]{K2017}). 
Combining the basic property of the Dirichlet series (see e.g. \cite[Theorem 1.1]{MV2007}) with it, we have Lemma \ref{lem:KE} (i). 

Next we prove Lemma \ref{lem:KE} (ii). 
For almost all $\omega \in \Omega$ satisfying Lemma \ref{lem:KE} (i), 
the equation
\begin{align}
\te_m(s, \omega) \nonumber
&= \lim_{X \rightarrow \infty} \sum_{2 \leq n \leq X} \frac{\Lambda (n) \omega (n) }{n^{s} (\log n)^{m +1} } \\
&= \lim_{X \rightarrow \infty} \left( \sum_{p \leq X} \frac{\omega(p)}{p^s (\log p)^m} + \sum_{p^k \leq X, k \geq2} \frac{ \omega (p)^k }{k^{m+1}p^{ks} (\log p)^m} \right)  \nonumber \\
&=  \sum_{p } \frac{\omega(p)}{p^s (\log p)^m} + \sum_{p} \sum_{k =2}^\infty \frac{ \omega (p)^k }{k^{m+1}p^{ks} (\log p)^m} \label{eqn:AS}
\end{align}
holds for $\Re(s) >\sigma_0(\mathcal{K})$ since the latter sum converges absolutely.  This gives the first assertion of Lemma \ref{lem:KE} (ii). 
On the other hand, 
the equation
\begin{align*}
\sum_{p}\frac{\Li_{m +1} ( p^{-s} \omega(p) ) }{(\log p)^m} 
&= \lim_{X \rightarrow \infty} \left( \sum_{p \leq X} \frac{\omega(p)}{p^s (\log p)^m} + \sum_{p \leq X} \sum_{k =2}^\infty \frac{ \omega (p)^k }{k^{m+1}p^{ks} (\log p)^m}  \right) \\
&= \sum_{p} \frac{\omega(p)}{p^s (\log p)^m} + \sum_{p} \sum_{k =2}^\infty \frac{ \omega (p)^k }{k^{m+1}p^{ks} (\log p)^m}
\end{align*}
holds for $\Re(s) > \sigma_0(\mathcal{K})$, 
because of the same reason. 
This gives the second assertion of Lemma \ref{lem:KE} (ii). 

Finally, we will show Lemma \ref{lem:KE} (iii) and (iv). 
Define the random variables
\[
S_\xi ( \omega)
= \sum_{2 \leq n \leq \xi} \frac{\Lambda(n) \omega(n) }{n^{\sigma_0(\mathcal{K})} (\log n)^{m + 1} } \quad \textrm{and} \quad
A_{\xi}(\omega)
= \sum_{p \leq \xi} \frac{ \omega (p) }{p^{\sigma_0(\mathcal{K})} (\log p)^m }
\]
for $\xi \geq 2$ and $\omega \in \Omega$.
By the almost all convergence of the series $\sum_p \omega (p) p^{- \sigma_0(\mathcal{K})} (\log p)^{ - m}$, 
we find that $\{ A_{\xi} (\omega) \}_{\xi \geq 2}$ is bounded almost surely.
We can write
\[
S_\xi (\omega) 
= A_{\xi}(\omega) + \sum_{p^k \leq \xi, k \geq 2} \frac{\omega(p)^k}{k^{m +1} p^{k \sigma_0 (\mathcal{K})} (\log p)^m },
\]
and the latter term is bounded.
Therefore we have, for almost all $\omega \in \Omega$, 
\begin{equation}\label{eqn:ASB}
|S_\xi (\omega)| \leq M(\mathcal{K}, \omega)
\end{equation}
for $\xi \geq 2$ with some positive constant $M(\mathcal{K}, \omega)$.
We fix $\omega \in \Omega$ which satisfies the inequality \eqref{eqn:ASB}. 
The estimate \eqref{eqn:ASB} and partial summation yield
\begin{equation}\label{eqn:KE}
\te_m(s, \omega)
= \int_{2-}^{\infty} \frac{d S_\xi(\omega)}{\xi^{s - \sigma_0(\mathcal{K})}} 
= (s - \sigma_0(\mathcal{K})) \int_{2}^\infty \frac{ S_\xi(\omega) d \xi}{\xi^{s + 1 - \sigma_0(\mathcal{K})}}
\ll_{\mathcal{K}, \sigma_L, \omega} | t | + 2
\end{equation} 
for $\Re(s) \geq 1/2 (\sigma_0(\mathcal{K}) + \sigma_L)$, which gives Lemma \ref{lem:KE} (iii). 
By using the equation \eqref{eqn:KE} and by the Cauchy-Schwarz inequality, 
the estimate
\[
\EXP^{\mathbf{m}}\left[| \te_m(s, \omega)| \right] 
\ll ( | t | + 2 ) \int_{2}^\infty \frac{ \EXP^{\mathbf{m}}\left[ |S_\xi(\omega) |\right] d \xi}{\xi^{\sigma + 1 - \sigma_0(\mathcal{K})}}
\leq ( | t | + 2 ) \int_{2}^\infty \frac{ \left(\EXP^{\mathbf{m}}\left[ |S_\xi(\omega) |^2 \right] \right)^{1/2} d \xi}{\xi^{\sigma + 1 - \sigma_0(\mathcal{K})}}
\]
holds for $\Re(s) \geq 1/2(\sigma_0(\mathcal{K}) + \sigma_L )$.
Since the estimate
\[
\EXP^{\mathbf{m}}\left[ |S_\xi(\omega) |^2 \right]
= \sum_{2 \leq n \leq \xi} \frac{\Lambda(n)^2}{n^{2\sigma_0(\mathcal{K})} (\log n)^{2m +1} }
\leq \sum_{n =2}^\infty \frac{\Lambda(n)^2}{n^{2\sigma_0(\mathcal{K})} (\log n)^{2m +1} }
\ll_{\mathcal{K}} 1
\]
holds for $\xi \geq 2$, 
we have the conclusion. 
\end{proof}

\begin{remark}
The assertion \rm{(i)} of this lemma will be used in subsection \ref{sb:SP}.
In proving the universality theorem, 
Rademacher-Menshov theorem (see e.g. \cite[Appendix B]{K2015}) is used ordinarily.
However we do not need to use it this time because the second term of \eqref{eqn:AS} is absolute convergent series.
\end{remark}

\begin{proof}[Proof of Lemma \ref{lem:LT2}]
Let $C$ be a compact subset of $\mathcal{R}$ and fix $\omega \in \Omega$ satisfying (ii) and (iii) of Lemma \ref{lem:KE}.
By the Mellin inversion formula of $\phi(x)$, which is stated in (iii) of Lemma \ref{lem:MT}, we have
\begin{align*}
\te_{m,X}(z,\omega) 
&= \frac{1}{2 \pi i} \sum_{n=2}^\infty \int_{c - i \infty}^{c + i \infty} \frac{\Lambda(n) \omega(n) \hat{\phi} (\xi) }{n^{\xi} (\log n)^{m + 1} } \left( \frac{n}{X} \right)^{- \xi} d\xi \\
&= \frac{1}{2 \pi i} \int_{c - i \infty}^{c + i \infty} \te_{m}(z + \xi,\omega) \hat{\phi} (\xi) X^\xi d\xi
\end{align*}
for $z \in \partial \mathcal{R}$ and for $c >1$,
where the interchange of the sum and the integral is justified by Fubini's theorem and (ii) of Lemma \ref{lem:MT}.
We will replace the contour from $c - i \infty$ to $c + i \infty$ by the one from $- \delta - i \infty$ to $- \delta + i \infty$ with $\delta = 1/4 ( \sigma_L - \sigma_0(\mathcal{K}) )$.
By our choice of $\phi(x)$ and by (i) of Lemma \ref{lem:MT}, 
we have
\begin{align*}
\te_{m}(z,\omega) - \te_{m,X}(z,\omega) 
= - \frac{1}{2 \pi i} \int_{- \delta - i \infty}^{- \delta + i \infty} \te_{m}(z + \xi,\omega) \hat{\phi} (\xi) X^\xi d\xi,
\end{align*}
where such a replacement is justified by (ii) of Lemma \ref{lem:MT} and (iii) of Lemma \ref{lem:KE}.
Using Cauchy's formula, we have
\begin{align*}
&\sup_{s \in C} \left| \te_m(s, \omega) - \te_{m,X} (s, \omega) \right| \\
\leq& \frac{1}{2 \pi \dist(C, \partial \mathcal{R})}
\int_{\partial \mathcal{R}} \left| \te_{m}(z,\omega) - \te_{m,X}(z,\omega) \right| | dz |\\
\leq& \frac{X^{- \delta}}{4 \pi^2 \dist(C, \partial \mathcal{R})} 
\int_{\partial \mathcal{R}} |dz| \times\\
& \times \left( \int_{ - \infty}^{ \infty} \left| \te_{m} \left(- \delta + \Re(z) + i \left( t + \Im(z) \right) ,\omega \right) \hat{\phi} (- \delta + it) \right| d t \right).
 \end{align*}
Taking the expectation of the above inequality, we obtain
\begin{align*}
&\EXP^{\mathbf{m}} \left[ \sup_{s \in C} \left| \te_m(s, \omega) - \te_{m,X} (s, \omega) \right| \right] \\
\leq& \frac{ \ell(\partial \mathcal{R}) X^{- \delta}}{4 \pi^2 \dist(C, \partial \mathcal{R})} 
\times  \\
& \times \sup_{z \in \partial \mathcal{R}} \left( \int_{ - \infty}^{ \infty} \EXP\left[ \left| \te_{m} \left(- \delta + \Re(z) + i \left( t + \Im(z) \right) ,\omega) \right) \right| \right]  \left| \hat{\phi} (- \delta + it) \right| d t \right) \\
\ll&_{\mathcal{K}, \mathcal{R}, C} X^{- \delta}
\rightarrow 0
\end{align*}
as $X \rightarrow \infty$ by the estimate (iv) of Lemma \ref{lem:KE} and (ii) of Lemma \ref{lem:MT}.
This completes the proof.
\end{proof}

We give one more lemma to prove the proposition.
Let $\mathcal{P}_0$ be a finite subset of the set of prime numbers. 
Define the probability measures $\mathbb{H}_T^{\mathcal{P}_0}$ on $\left(\prod_{p \in \mathcal{P}_0} \gamma_p, \mathcal{B}\left( \prod_{p \in \mathcal{P}_0} \gamma_p \right) \right)$ by
\[
\mathbb{H}_T^{\mathcal{P}_0}(A)
= \frac{1}{\meas\left(\mathcal{I}_{\mathcal{K}}(T) \right)}
\meas \left\{ \tau \in \mathcal{I}_{\mathcal{K}}(T) ~;~ (p^{i \tau})_{p \in \mathcal{P}_0} \in A \right\}
\]
for $A \in \mathcal{B}\left( \prod_{p \in \mathcal{P}_0} \gamma_p \right)$ and write $\mathbf{m}_{\mathcal{P}_0} = \otimes_{p \in \mathcal{P}_0} \mathbf{m}_{p}$.
Then we have the following lemma.

\begin{lemma}\label{lem:WC}
The probability measure $\mathbb{H}_T^{\mathcal{P}_0}$ converges weakly to $\mathbf{m}_{\mathcal{P}_0}$ as $T \rightarrow \infty$.
\end{lemma}

\begin{proof}
It is enough to check that the Fourier transform of $\mathbb{H}_T^{\mathcal{P}_0}$ converges pointwise to that of $\mathbf{m}_{\mathcal{P}_0} $ as $T \rightarrow \infty$.
Let $\mathcal{F}(~\cdot~;\mathbb{H}_T^{\mathcal{P}_0})$ denote the Fourier transform of $\mathbb{H}_T^{\mathcal{P}_0}$. 
For any $\mathbf{n}= (n_p)_{p \in \mathcal{P}_0} \in \mathbb{Z}^{\mathcal{P}_0}$, 
we have
\begin{equation*}
\mathcal{F}(\mathbf{n};\mathbb{H}_T^{\mathcal{P}_0})
=\int_{\underset{p \in \mathcal{P}_0}{\Pi} \gamma_p} \prod_{p \in \mathcal{P}_0} x_p^{n_p} ~ d \mathbb{H}_T^{\mathcal{P}_0}\left((x_p)_{p \in \mathcal{P}_0}\right)
=\frac{1}{\meas\left(\mathcal{I}_{\mathcal{K}}(T) \right)} 
\int_{\mathcal{I}_{\mathcal{K}}(T)} \prod_{p \in \mathcal{P}_0} p^{i n_p \tau} d \tau. 
\end{equation*}
If $\mathbf{n} = \mathbf{0}$, then we have $\lim_{T \rightarrow \infty} \mathcal{F}(\mathbf{0};\mathbb{H}_T^{\mathcal{P}_0})=1$.
If $\mathbf{n} \neq \mathbf{0}$, we have
\begin{align*}
&\mathcal{F}(\mathbf{n};\mathbb{H}_T^{\mathcal{P}_0})
=\frac{1}{\meas\left(\mathcal{I}_{\mathcal{K}}(T) \right)}
\int_{T}^{2T} \exp\left( i \tau \sum_{p \in \mathcal{P}_0} n_p \log p \right) d \tau
+ O\left(\frac{T - \meas\left(\mathcal{I}_{\mathcal{K}}(T)\right)}{\meas\left(\mathcal{I}_{\mathcal{K}}(T) \right)}\right) \\
&= \frac{1}{\meas \left(\mathcal{I}_{\mathcal{K}}(T) \right)} \left[ \frac{ \exp\left( i \tau \sum_{p \in \mathcal{P}_0} n_p \log p \right)}{ i \sum_{p \in \mathcal{P}_0} n_p \log p} \right]_{\tau = T}^{\tau = 2T} + O\left(\frac{T - \meas\left(\mathcal{I}_{\mathcal{K}}(T)\right)}{\meas\left(\mathcal{I}_{\mathcal{K}}(T) \right)}\right).
\end{align*}
The above first term tends to $0$ trivially and so does the second term by the estimate $\meas\left(\mathcal{I}_{\mathcal{K}}(T)\right) \sim T$. 
This completes the proof. 
\end{proof}

\begin{proof}[Proof of Proposition \ref{prop:LT}]
Let $\mathbb{P}_T$ denote the probability measure on $(\mathcal{I}_{\mathcal{K}}(T), \mathcal{B}(\mathcal{I}_{\mathcal{K}}(T)) )$ given by
\[
\mathbb{P}_T(E) = \frac{1}{\meas\left(\mathcal{I}_{\mathcal{K}}(T) \right)} \meas(E), \quad E \in \mathcal{B}(\mathcal{I}_{\mathcal{K}}(T) ).
\]
Let $F:\mathcal{H}(\mathcal{R}) \rightarrow \mathbb{R}$ be a bounded and Lipschitz continuous function. 
Then there exist $M(F),C(F)\geq 0$ such that the estimates
\[
| F(f) | \leq M(F) \AND | F ( f ) - F (g) | \leq C(F) d (f,g)
\] 
hold for any $f,g \in \mathcal{H}(\mathcal{R})$.
It suffices to show that $\left| \EXP^{\mathcal{Q}_T} \left[ F \right] - \EXP^{\mathcal{Q}} \left[ F \right] \right| \rightarrow 0$ as $T \rightarrow \infty$ by the property of weak convergence (see e.g. \cite[Theorem 3.9.1]{D2010}).
We find that
\begin{align*}
&\left| \EXP^{\mathcal{Q}_T} \left[ F \right] - \EXP^{\mathcal{Q}} \left[ F \right] \right|
= \left| \EXP^{\mathbb{P}_T} \left[ F(\te_m(s + i\tau)) \right] - \EXP^{\mathbf{m}} \left[ F ( \te_m(s, \omega) ) \right] \right| \\
\leq & \left| \EXP^{\mathbb{P}_T} \left[ F(\te_m(s + i\tau)) \right] - \EXP^{\mathbb{P}_T} \left[ F(\te_{m,X}(s + i\tau)) \right]  \right| \\
&+ \left| \EXP^{\mathbb{P}_T} \left[ F(\te_{m,X}(s + i\tau)) \right] - \EXP^{\mathbf{m}} \left[ F ( \te_{m,X} (s, \omega) ) \right] \right| \\
& \quad + \left| \EXP^{\mathbf{m}} \left[ F ( \te_{m,X} (s, \omega) ) \right] - \EXP^{\mathbf{m}} \left[ F ( \te_m(s, \omega) ) \right] \right| \\
=:& \Sigma_1 (T,X) + \Sigma_2 (T,X) + \Sigma_3 (T,X).
\end{align*}
First we estimate $\Sigma_1 (T,X)$.
We have
\begin{align*}
\Sigma_1 (T,X)
\leq&  \frac{M(F) \meas \left(\mathcal{I}_{\mathcal{K}}(T) \setminus \mathscr{X}_{\mathcal{K}}( T ) \right) }{\meas \left(\mathcal{I}_{\mathcal{K}}(T) \right)} \\
&+ C(F)  \EXP^{\mathbb{P}_T} \left[ d \left( \te_m(s + i\tau), \te_{m,X}(s + i\tau) \right) ; \tau \in \mathscr{X}_{\mathcal{K}}( T ) \right],
\end{align*}
where $\EXP[X; A] $ stands for $\EXP[X \mathbbm{1}_A]$ for a random variable $X$ and an indicator function $\mathbbm{1}_A$ of the set $A$.
The first term of the above inequality tends to $0$ as $T \rightarrow \infty$, 
since 
\[
\mathscr{X}_{\mathcal{K}}(T) \subset \mathcal{I}_{\mathcal{K}}(T) \AND 
\lim_{T \rightarrow \infty}\frac{\meas\left( \mathscr{X}_{\mathcal{K}}(T) \right)}{T}
= \lim_{T \rightarrow \infty}\frac{\meas\left( \mathcal{I}_{\mathcal{K}}(T) \right)}{T}
=1.
\]
As for the second term, we have
\begin{align*}
&\EXP^{\mathbb{P}_T} \left[ d \left( \te_m(s + i\tau), \te_{m,X}(s + i\tau) \right) ; \tau \in \mathscr{X}_{\mathcal{K}}( T ) \right] \\
\leq& \sum_{j =1}^\infty \frac{\min\left\{\EXP^{\mathbb{P}_T} \left[ d_j \left( \te_m(s + i\tau), \te_{m,X}(s + i\tau) \right) ; \tau \in \mathscr{X}_{\mathcal{K}}( T ) \right], 1\right\}}{2^j}.
\end{align*}
By Lemma \ref{lem:LT1} and Lebesgue's dominated convergence theorem, 
this also tends to $0$ as we take $\lim_{X \rightarrow \infty} \limsup_{T \rightarrow \infty}$.

Next we estimate $\Sigma_2 (T,X)$.
Put 
\[
\mathcal{P}(\phi, X)
=\left\{ p~;~ \textrm{$p$ divides $\displaystyle \prod_{ \substack{n \in \mathbb{N};\\ \phi(n/X) \neq 0 } } n$ } \right\}
\]
and define the continuous mapping $\Phi_{m,X}:\prod_{p \in \mathcal{P}(\phi, X)} \gamma_p \rightarrow \mathcal{H} ( \mathcal{R} )$ by
\[
\Phi_{m,X}(x) = \sum_{\substack{n = 2 \\ \phi(n /X) \neq 0 }}^{\infty} \frac{\Lambda (n) \phi (n /X) }{n^s (\log n)^{m +1}} \prod_{p|n} x_p^{- i \tau \nu (p;n)}, \quad
x = (x_p)_{p \in \mathcal{P}(\phi, X)} \in \prod_{p \in \mathcal{P}(\phi, X)} \gamma_p,
\]
where $\nu(p;n)$ is the exponent of $p$ in the prime factorization of $n$.
Then we have
\begin{equation*}
\EXP^{\mathbb{P}_T} \left[ F(\te_{m,X}(s + i\tau)) \right]
= \EXP^{\mathbb{H}_T^{\mathcal{P}(\phi, X)} \circ \Phi_{m,X}^{-1} }\left[ F \right]
\rightarrow \EXP^{\mathbf{m}_{\mathcal{P}(\phi, X)} \circ \Phi_{m,X}^{-1} }\left[ F \right]
= \EXP^{\mathbf{m}} \left[ F ( \te_{m,X} (s, \omega) ) \right]
\end{equation*}
as $T \rightarrow \infty$ by Lemma \ref{lem:WC} and the property of weak convergence (see e.g. \cite[Section 2, The Mapping Theorem]{B1999}).
Hence we obtain $\lim_{T \rightarrow} \Sigma_2 (T,X) = 0$.

Finally we estimate $\Sigma_3 (T,X)$.
We have
\[
\Sigma_3 (T,X)
\leq C(F) \sum_{j =1}^\infty \frac{\min \left\{ \EXP^{\mathbf{m}} \left[ d_j \left( \te_m(s, \omega) ,\te_{m,X}(s , \omega) \right) \right], 1 \right\}}{2^j},
\]
and the right hand side of the above inequality tends to $0$ as $X \rightarrow \infty$ by Lemma \ref{lem:LT2} and Lebesgue's dominated convergence theorem.
This completes the proof.
\end{proof}

\subsection{Proof of Proposition \ref{prop:SP}}\label{sb:SP}
We will prove the proposition by using the standard method as in \cite{L1996} and \cite{S2007}.
First we recall the following lemmas.

\begin{lemma}\label{lem:DLH}
Let $D$ be simply connected region in the complex plane.
Suppose that the sequence $\{f_n\}_{n =1}^\infty$ in $\mathcal{H}(D)$ satisfies the following assumptions.
\begin{itemize}
\item[(a)] If $\mu$ is a complex Borel measure on $(\mathbb{C}, \mathcal{B}(\mathbb{C}))$ with compact support contained in $D$ such that
\[
\sum_{n =1}^{\infty} \left| \int_{\mathbb{C}} f_n d \mu \right| < \infty,
\]
then
\[
\int_{\mathbb{C}} s^r d \mu (s) = 0
\]
for any $r = 0, 1, 2, \ldots$.
\item[(b)] The series $\sum_{n =1}^\infty f_n$ converges in $\mathcal{H}(D)$.
\item[(c)] For every compact set $K \subset D$, 
\[
\sum_{n =1}^\infty \max_{s \in K} | f_n (s) |^2 < \infty.
\]
\end{itemize}
Then the set of all convergent series 
\[
\sum_{n = 1}^{\infty} c(n) f_n, \quad c(n) \in \gamma
\]
is dense in $\mathcal{H}(D)$.
\end{lemma}

\begin{proof}
This can be found in \cite[Theorem 3.10]{L1996} and \cite[Theorem 5.7]{S2007}.
\end{proof}

\begin{lemma}\label{lem:ET}
Let $\mu$ be a complex Borel measure on $(\mathbb{C}, \mathcal{B}(\mathbb{C}))$ with compact support contained in the half-plane $\sigma> \sigma_0$.
Moreover, for $s \in \mathbb{C}$, define the function
\[
f(s) = \int_{\mathbb{C}} \exp ( s z ) d \mu (z).
\]
Then $f$ is an entire function of exponential type. 
If $f(s)$ does not vanish identically, then
\[
\limsup_{r \rightarrow \infty} \frac{ \log | f (r) | }{r} > \sigma_0.
\]
\end{lemma}

\begin{proof}
This can be found in \cite[Lemma 4.10]{L1996}.
\end{proof}

\begin{lemma}\label{lem:BT}
Let $f(s)$ be an entire function of exponential type, and let $\{ \xi_m \}_{m =1}^\infty$ be a sequence of complex numbers.
Moreover, assume that there are positive real constants $\lambda$, $\eta$, and $\omega$ such that
\begin{itemize}
\item[(a)] $\limsup_{y \rightarrow \infty} \frac{\log | f ( \pm i y ) |}{y} \leq \lambda$,
\item[(b)] $| \xi_m - \xi_n | \geq \omega |m - n|$ for any $m, n \in \mathbb{N}$,
\item[(c)] $\lim_{m \rightarrow \infty} \frac{\xi_m}{m} = \eta$,
\item[(d)] $\lambda \eta < \pi$.
\end{itemize}
Then
\[
\limsup_{m \rightarrow \infty} \frac{ \log | f (\xi_m) | }{ | \xi_m | }
= \limsup_{r \rightarrow \infty} \frac{ \log | f (r) |  }{r}.
\]
\end{lemma}

\begin{proof}
This can be proved by using Bernstein's theorem. 
For the proof of Bernstein's theorem, see e.g. \cite{L1936}.
For the proof of this lemma by using Bernstein's theorem, see e.g. \cite[Theorem 4.12]{L1996}.
\end{proof}

\begin{lemma}\label{lem:KDL}
Let $N$ be a positive integer.
The set of all convergent series
\[
\sum_{p > N} \frac{\omega(p)}{p^s (\log p)^m}, \quad \omega(p) \in \gamma
\]
is dense in $\mathcal{H}(\mathcal{R})$.
\end{lemma}

\begin{proof}
By (i) of Lemma \ref{lem:KE}, we can take a convergent series 
\[
\sum_{p >N} \frac{\tilde{\omega} (p)}{ p^{ s } (\log p)^{m}}
\]
with some $\tilde{\omega} (p) \in \gamma$.
We will check that the sequence $\{ \tilde{\omega} (p) p^{ - s } (\log p)^{-m} \}_{p >N}$ satisfies the assumption (a)-(c) of Lemma \ref{lem:DLH} with $D$ replaced by $\mathcal{R}$.
It is obvious that the sequence satisfies the assumption (b) of Lemma \ref{lem:DLH}.

For any compact set $C \subset \mathcal{R}$, the estimate
\[
\sum_{p > N} \max_{s \in C} |  \tilde{\omega} (p) p^{ - s } (\log p)^{-m} |^2 
\leq \sum_{p} p^{- 2 \sigma_0(\mathcal{K})} (\log p)^{-2m} < \infty
\]
holds, which gives the assumption (c) of Lemma \ref{lem:DLH}.

We will see the sequence $\{ \tilde{\omega} (p) p^{ - s } (\log p)^{-m} \}_{p >N}$ satisfies the assumption (a) of Lemma \ref{lem:DLH} by using Lemma \ref{lem:ET} and Lemma \ref{lem:BT}.
Let $\mu$ be a complex Borel measure on $(\mathcal{C}, \mathcal{B}(\mathbb{C}))$ with compact support contained in $\mathcal{R}$ which satisfies
\begin{equation}\label{eqn:DLE1}
\sum_{p > N} \left| \int_{\mathbb{C}}  \tilde{\omega} (p) p^{ - s } (\log p)^{-m} d \mu (s) \right|
< \infty.
\end{equation}
Put 
\[
\rho (z) 
= \int_{\mathbb{C}} \exp ( - s z ) d \mu (s).
\]
We will show $\rho(z) \equiv 0$. 
By the equation \eqref{eqn:DLE1}, 
we find that
\begin{equation}\label{eqn:DLE2}
\sum_{p} \frac{ | \rho ( \log p ) |}{ (\log p)^m } < \infty.
\end{equation}

We take a large positive number $M > 0$ such that the support of $\mu$ is contained in the region $\{  s \in \mathcal{R}; | t | < M -1 \}$.
By the definition of $\rho(z)$, we find that
\[
\left| \rho( \pm i y ) \right| 
\leq \exp( M y ) |\mu|(\mathbb{C})
\]
for $y >0$, where $| \mu | $ stands for the total variation of $\mu$.
Hence we have
\begin{equation}\label{eqn:BT1}
\limsup_{y \rightarrow \infty} \frac{ \log \left| \rho ( \pm i y ) \right| }{y} \leq M.
\end{equation}
Fix a positive number $\eta$ satisfying 
\begin{equation}\label{eqn:BT2}
M \eta < \pi
\end{equation}
and let $A$ denote the set of all positive integers $n$ such that there exists a positive  number $r \in ( (n - 1/4) \eta, (n + 1/4)\eta ] $ such that $| \rho (r) | \leq \exp ( - \sigma_R r )$.
For any positive integer $n$, put
\[
\alpha_n = \exp ( (n - 1/4) \eta ) 
\quad \textrm{and} \quad
\beta_n = \exp( (n+1/4) \eta ).
\]
Note that $|\rho(\log p)| > p^{- \sigma_R}$ holds for any $n \not \in A$ and $\alpha_n < p \leq \beta_n$ by the definition of $A$.
Then we have
\begin{align*}
&\sum_{p} \frac{ | \rho ( \log p ) |}{ (\log p)^m }
\geq \sum_{n =1}^\infty \sum_{\alpha_n < p \leq \beta_n} \frac{ | \rho ( \log p ) |}{ (\log p)^m }
\geq \sum_{n \not\in A} \sum_{\alpha_n < p \leq \beta_n} \frac{ | \rho ( \log p ) |}{ (\log p)^m }\\
\geq& \sum_{n \not\in A} \sum_{\alpha_n < p \leq \beta_n} \frac{1}{ p^{\sigma_R} (\log p)^m }
\geq \sum_{n \not \in A} \frac{1}{\beta_n^{\sigma_R} (\log \beta_n)^m } \left( \pi(\beta_n) - \pi(\alpha_n) \right).
\end{align*}
Using the prime number theorem $\pi(x) = \int_{2}^x \frac{du}{\log u} + O \left(x \exp\left( - c \sqrt{\log x} \right) \right)$ with some constant $c>0$,
we have
\begin{align*}
\pi(\beta_n) - \pi(\alpha_n)
&= \int_{\alpha_n}^{\beta_n} \frac{ du }{\log u} + O \left( \beta_n \exp( - c \sqrt{ \log \alpha_n } ) \right) \\
&\geq \frac{ \beta_n ( 1 - e^{ - \eta/2 } ) }{\log \beta_n} + O \left( \beta_n \exp\left( - (c/2) \sqrt{\log \beta_n} \right) \right)
\gg _\eta\frac{ \beta_n}{\log \beta_n}
\end{align*}
for sufficiently large $n \in \mathbb{N}$.
Hence we have
\begin{align*}
\sum_{p} \frac{ | \rho ( \log p ) |}{ (\log p)^m } 
&\gg_{\eta} \sum_{n \not \in A; n \geq n_0} \frac{\beta_n^{1 - \sigma_R}}{(\log \beta_n)^{m + 1} } \\
&\gg_{m, \eta} \sum_{n \not \in A; n \geq n_0} \frac{\exp\left( ( 1 - \sigma_R ) n \eta \right)}{n^{m+1}} 
\gg \sum_{n \not \in A; n \geq n_0} 1
\end{align*}
with some large constant $n_0 \in \mathbb{N}$.
Combining the above estimate with the inequality \eqref{eqn:DLE2}, we find that 
\[
\sum_{n \not \in A; n \geq n_0} 1 < \infty.
\]
Therefore there exists a positive constant $n_1 \in \mathbb{N}$ such that $\{ n \in \mathbb{N}; n \geq n_1 \} \subset A$.
By the definition of $A$, 
we can take a sequence $\{ \xi_n \}_{n \geq n_1}$ so that
\[
(n - 1/4) \eta < \xi_n \leq (n + 1/4) \eta \quad 
\textrm{and} \quad
| \rho(\xi_n) | \leq \exp ( - \sigma_R \xi_n ).
\]
From these, we have
\begin{equation}\label{eqn:BT3}
\lim_{n \rightarrow \infty} \frac{\xi_n}{n} = \eta 
\quad \textrm{and} \quad
\limsup_{n \rightarrow \infty} \frac{ \log | \rho ( \xi_n ) | }{\xi_n} \leq - \sigma_R,
\end{equation}
and
\begin{equation}\label{eqn:BT4}
\xi_m - \xi_n 
\geq  ( m - 1/4 ) \eta - ( n + 1/4 ) \eta
= ( m - n ) \eta - \eta/2 
\geq \eta/2 ( m -n )
\end{equation}
for $m > n \geq n_1$.
Applying Lemma \ref{lem:BT} with \eqref{eqn:BT1}, \eqref{eqn:BT2}, \eqref{eqn:BT3}, \eqref{eqn:BT4}, 
we have
\begin{equation}\label{eqn:BT5}
\limsup_{r \rightarrow \infty} \frac{ \log | \rho (r) | }{r} \leq - \sigma_R.
\end{equation}
On the other hand, assuming that $\rho (z)$ does not vanish identically, 
\[
\limsup_{r \rightarrow \infty} \frac{ \log | \rho (r) | }{r} > - \sigma_R.
\]
holds by Lemma \ref{lem:ET}, which contradicts the inequality \eqref{eqn:BT5}.
Hence we obtain $\rho (z) \equiv 0$.
Since $\rho(z)$ is represented as
\[
\rho(z) = \sum_{r = 0}^\infty \left( \frac{(-1)^r}{r !} \int_{\mathbb{C}} s^r d \mu(s) \right) z^r,
\]
we conclude that $\int_{\mathbb{C}} s^r d \mu(s)=0$ for any $r = 0, 1, 2, \ldots$ by the uniqueness of the Taylor expansion, which gives the assumption (a) of Lemma \ref{lem:DLH}. 
This completes the proof. 
\end{proof}

Using this lemma, we show the following lemma.

\begin{lemma}\label{lem:DL}
The set of all convergent series
\[
\sum_p \frac{\Li_{m +1} ( p^{ - s} \omega(p) )}{ (\log p)^{m}}
\]
is dense in $\mathcal{H}(\mathcal{R})$.
\end{lemma}

\begin{proof}
Fix $f \in \mathcal{H}(\mathcal{R})$ and $\epsilon>0$.
We put
\[
h_N(s,\omega_N)
= \sum_{p>N} \left(  \frac{\Li_{m +1}(p^{-s} \omega(p))}{(\log p)^m} - \frac{\omega(p)}{p^s ( \log p )^m} \right)  = \sum_{p > N} \sum_{k \geq 2} \frac{\omega(p)^k}{k^{m +1} p^{ks} ( \log p )^m },
\]
for $s \in \mathcal{R}$ and $\omega_N = (\omega(p))_{ p > N} \in \prod_{p > N} \gamma_p $, 
which converges absolutely.
Then we have
\[
\| h_N \|_\infty:=
\sup_{\omega_N \in \underset{p > N}{\Pi} \gamma_p } \sup_{s \in \mathcal{R}} \left| h_N(s,\omega_N) \right|
\leq \sum_{p > N} \sum_{k \geq 2} \frac{1}{k^{m +1} p^{k \sigma_L} (\log p)^m }
\rightarrow 0
\]
as $N \rightarrow \infty$.
Hence we can take a large $N_0 = N_0(\epsilon)$ so that $\| h_{N_0} \|_\infty < \epsilon/2$.
By Lemma \ref{lem:KDL}, there exists a convergent series 
\[
\sum_{p > N_0} c_0(p) p^{ - s} (\log p)^{- m}
\] with some $c_0 (p) \in \gamma$, $p > N_0$ such that
\[
d \left ( f(s) - \sum_{p \leq N_0} \frac{\Li_{m +1} ( p^{- s}  ) }{ ( \log p )^m }, \sum_{p > N_0} \frac{c_0(p)}{ p^s ( \log p )^m } \right) 
< \epsilon/2.
\]
Putting
\[
c(p)
=
\begin{cases}
1 & \textrm{if $p \leq N_0$}, \\
c_0(p)  & \textrm{if $p > N_0$},
\end{cases}
\]
we obtain
\begin{align*}
&d \left( f(s), \frac{\Li_{m +1} ( p^{-s
} c(p) ) }{ ( \log p )^m } \right) \\
\leq& d \left ( f(s) - \sum_{p \leq N_0} \frac{\Li_{m +1} ( p^{ -s}  ) }{ ( \log p )^m }, \sum_{p > N_0} \frac{c_0(p)}{ p^s ( \log p )^m } \right) \\
& \quad \quad + d \left( \sum_{p > N_0} \frac{\Li_{m +1} ( p^{-s
} c(p) ) }{ ( \log p )^m }, \sum_{p > N_0}\frac{c(p)}{p^s ( \log p )^m} \right) 
\leq \epsilon/2 + \| h_{N_0}\|_\infty < \epsilon.
\end{align*}
This completes the proof.
\end{proof}

\begin{lemma}\label{lem:ASCS}
Let $(X_n)_{n=1}^\infty$ be a sequence of $\mathcal{H}(\mathcal{R})$-valued independent random variables such that the series $X = \sum_{n=1}^\infty X_n$ converges almost surely. Then the support of $X$ is the closure of the set of all convergent series of the form $\sum_{n =1}^\infty x_n$, where $x_n$ belongs to the support of the distribution of $X_n$ for all $n \geq 1$.
\end{lemma}

\begin{proof}
The proof can be found in \cite[Appendix B]{K2015}.
\end{proof}

\begin{proof}[Proof of Proposition \ref{prop:SP}]
By (ii) of Lemma \ref{lem:KE}, 
we find that the distribution of $\te_m(s,\omega)$ equals that of $\sum_{p} \Li_{m +1} ( p^{ - s} \omega(p) )(\log p)^{-m}$. 
We will confirm that the support of the distribution of $\Li_{m +1} ( p^{ - s} \omega(p) )(\log p)^{-m}$ equals the set of functions $\Li_{m +1} ( p^{ - s} z )(\log p)^{-m}$, $z \in \gamma$.
We note that one has the estimate
\begin{equation}\label{eqn:EPL}
\left| \Li_{m +1}(w_1) - \Li_{m +1}(w_2) \right|
\leq \frac{| w_1- w_2 |}{ 1 - \max\left\{ |w_1| , |w_2| \right\} }
\end{equation}
for any $|w_1|, | w_2 | < 1$. 
First we show the set of functions $\Li_{m +1} ( p^{ - s} z)(\log p)^{-m}$, $z \in \gamma$ is included in the support of the distribution of $\Li_{m +1} ( p^{ - s} \omega(p) )(\log p)^{-m}$. 
We fix a function $\Li_{m +1} ( p^{ - s} z)(\log p)^{-m}$, $z \in \gamma$.
Then we have
\begin{align*}
& \mathbf{m} \left( d \left( \frac{\Li_{m +1} ( p^{ - s} \omega(p) )}{ (\log p)^{m}}, \frac{\Li_{m +1} ( p^{ - s} z)}{(\log p)^{m} }\right) < \epsilon \right)\\
=& \frac{1}{2 \pi} \int_0^{2 \pi} \mathbbm{1} \left\{ d \left( \frac{\Li_{m +1} ( p^{ - s} e^{i \theta})}{(\log p)^{m}}, \frac{\Li_{m +1} ( p^{ - s} z)}{(\log p)^{m}} \right) < \epsilon \right\} d \theta \\
\geq& \frac{1}{2 \pi} \int_0^{2 \pi} \mathbbm{1} \left\{ \frac{| e^{i \theta} -  z |}{ \left(p^{\sigma_L} -1 \right) ( \log p )^m  } < \epsilon \right\} d \theta
> 0
\end{align*}
by the inequality \eqref{eqn:EPL}, which gives the desired inclusion. 

We will prove the converse inclusion.  
Let $g(s)$ be a holomorphic function in $\mathcal{R}$ which is not included in the set of functions $\Li_{m +1} ( p^{ - s} z)(\log p)^{-m}$, $z \in \gamma$.
Since the mapping 
\[
\gamma \ni z \mapsto d \left( \Li_{m +1} ( p^{ - s} z)(\log p)^{-m}, g(s) \right) \in \mathbb{R}
\]
is continuous, 
we have 
\[
\epsilon_0: =\min_{z \in \gamma} d \left( \Li_{m +1} ( p^{ - s} z)(\log p)^{-m}, g(s) \right) >0.
\]
Hence we obtain
\[
\mathbf{m} \left( d \left( \Li_{m +1} ( p^{ - s} \omega(p) )(\log p)^{-m}, g ( s ) \right) < \epsilon_0 \right)=0.
\]
Therefore $g (s) $ does not belong to the support of the distribution of $\Li_{m +1} ( p^{ - s} \omega(p) )(\log p)^{-m}$.
Thus Lemma \ref{lem:ASCS} and Lemma \ref{lem:DL} yield the conclusion.
\end{proof}

\subsection{Proof of Theorem \ref{thm:MTH}}
Let $m$ be a non-negative integer, let $\mathcal{K}$ be a compact subset of $\mathcal{D}$ with connected complement, and let $\mathcal{R}$ be the rectangle in $\mathbb{C}$ given by \eqref{eqn:RECT}.
Suppose that $f(s)$ is a continuous function on $\mathcal{K}$ that is holomorphic in the interior of $\mathcal{K}$.
Fix $\epsilon>0$. 
By Mergelyan's approximation theorem (see e.g. \cite[Theorem 20.5]{R1987}), there exists a polynomial $P(s)$ such that
\[
\sup_{s \in \mathcal{K}} \left| P(s) - f (s) \right| < \epsilon/2.
\]
Let $\Phi(P)$ be the set of all holomorphic functions $g$ in $\mathcal{H}(\mathcal{R})$ which satisfy $\sup_{s \in \mathcal{K}} | g(s) - P(s) | < \epsilon/2$.
Then we find that the set $\Phi(P)$ is open in $\mathcal{H}(\mathcal{R})$.
By Proposition \ref{prop:LT}, Proposition \ref{prop:SP} and the Portmanteau theorem (see e.g. \cite[Theorem 2.1]{B1999} ), we have
\begin{align*}
&\liminf_{T \rightarrow \infty} \frac{1}{T} \meas\left\{ \tau \in [T,2T] ~;~ \sup_{s \in \mathcal{K}} \left| \te_{m} (s + i \tau) - P(s) \right| < \epsilon/2 \right\}\\
\leq& \liminf_{T \rightarrow \infty} \frac{1}{\meas\left( \mathcal{I}_{\mathcal{K}}(T) \right)} \meas\left\{ \tau \in \mathcal{I}_{\mathcal{K}}(T); \sup_{s \in \mathcal{K}} \left| \te_{m} (s + i \tau) - P(s) \right| < \epsilon/2 \right\} \\
=& \liminf_{T \rightarrow \infty} \mathcal{Q}_T ( \Phi(P) )
\geq \mathcal{Q} ( \Phi(P) ) > 0.
\end{align*}
The above inequality and the inequality
\[
\sup_{s \in \mathcal{K}} \left| \te_{m} (s + i \tau) - f (s) \right|
\leq \sup_{s \in \mathcal{K}} \left| \te_{m} (s + i \tau) - P(s) \right| + \sup_{s \in \mathcal{K}} \left| P(s) - f (s) \right|
\]
finish the proof.

\begin{ackname}
The author would like to deeply thank Mr.~Shota Inoue and Mr.~Mine Masahiro for giving me some valuable advise. 
The author also would like to Professor Kohji Matsumoto for his helpful comments.
\end{ackname}

\begin{flushleft}
{\footnotesize
{\sc
Graduate School of Mathematics, Nagoya University,\\
Chikusa-ku, Nagoya 464-8602, Japan.
}\\
{\it E-mail address}, K. Endo: {\tt m16010c@math.nagoya-u.ac.jp}
}
\end{flushleft}

\end{document}